\documentclass[10pt,a4paper,oneside]{article}
\usepackage{amsmath}
\usepackage{amssymb}
\usepackage{amsthm}
\usepackage{mathrsfs}
\usepackage[usenames]{color}

\newcommand{\ovr}{\overline{r}}

\newcommand{\ovR}{\overline{R}}
\newcommand{\e}{\varepsilon}
\newtheorem{theorem}{Theorem}

\newtheorem{proposition}[theorem]{Proposition}
\newtheorem{lemma}[theorem]{Lemma}

\def\H{\mathbb{H}}
\def\R{\mathbb{R}}

\def\length{\mathrm{length}}

\def \eps {\epsilon}

\def \P {{\bf P}}

\def \E {{\bf E}}

\def \_reg {\rightarrow_{\bf reg}}

\def\maxdeg/{\Delta}

\def\area{\mathop{\mathrm{area}}}

\def\length{\mathop{\mathrm{length}}}

\def\OC{\mathcal C}
\def\VC{\mathcal W}

\title{Asymptotics of visibility in the hyperbolic plane}
\author{Johan Tykesson\thanks{The Weizmann Institute of Science,
    Faculty of Mathematics and Computer Science, POB 26, Rehovot 76100, Israel. E-mail: {\tt
      johan.tykesson@gmail.com}. Research supported by a post-doctoral grant of the Swedish
     Research Council.} \and Pierre Calka\thanks{Laboratoire de Math\'ematiques Rapha\"el Salem, UMR 6085 CNRS-Universit\'e de Rouen, Avenue de l'Universit\'e, BP.12 Technop\^ole du Madrillet F76801 Saint-Etienne-du-Rouvray, France. E-mail: {\tt pierre.calka@univ-rouen.fr}.}}
\begin{document}
\maketitle
\footnotetext{{\em American Mathematical Society 2000 subject
classifications.} Primary 60D05, Secondary 60G55} 
\begin{abstract}
At each point of a Poisson point process of intensity
$\lambda$ in the hyperbolic place, center a ball of bounded random
radius. Consider the probability $P_r$ that from a fixed point, there is
some direction in which one can reach distance $r$ without hitting
any ball. It is known \cite{BJST} that if $\lambda$ is strictly
smaller than a critical intensity $\lambda_{gv}$ then $P_r$ does
not go to $0$ as $r\to \infty$. The main result in this note shows
that in the case $\lambda=\lambda_{gv}$, the probability of
reaching distance larger than $r$ decays essentially polynomial,
while if $\lambda>\lambda_{gv}$, the decay is exponential. We also
extend these results to various related models and we finally obtain asymptotic results in several situations.
\end{abstract}
\section{Introduction}\label{intro}

Let $X$ be a homogeneous Poisson point process with intensity
$\lambda\in (0,\infty)$ in the hyperbolic plane $\H^2$. At each
point of $X$, center a ball of a bounded random radius,
independently for all points. Fix a base-point $o\in \H^2$. In
\cite{BJST}, it was shown that there is a critical intensity
$\lambda_{gv}\in (0,\infty)$, such that if $\lambda<
\lambda_{gv}$, then with positive probability there is some
geodesic ray, starting at $o$, such that it does not hit any of
the balls. In other words, if you stand at $o$, then with positive
probability you have visibility to infinity inside the complement
of the balls in some direction. Of course, such a direction much
be exceptional, since in a given direction, you will hit
infinitely many balls with probability one.

In \cite{BJST} it was also shown that as soon as $\lambda\ge
\lambda_{gv}$, with probability one, there is no direction in
which you can see to infinity. In other words, the set of visible
points from $o$ are with probability $1$ within some finite random
distance. In this note we mainly investigate the probability that
there is some direction in which you can see to a distance larger than
$r$ inside the complement of the balls, when $\lambda\ge
\lambda_{gv}$. In this region, the probability to see distance
larger than $r$ in some direction goes to $0$ as $r$ approaches
infinity, and here we are interested in what the decay looks like
for large $r$. We will see that at the critical value, this decay
is essentially polynomial while above criticality, the decay is
essentially exponential which is different from the decay for the visibility in a fixed direction. This also differs from the Euclidean case,
where Calka et. al. \cite{CMPB} showed that for every $\lambda>0$,
one has exponential decay and the same decay as for the visibility in a fixed direction. We also generalize these results to
visibility outside a Poisson process on the space of lines in
$\H^2$. Indeed, Benjamini et al. \cite{BJST} extended a previous work due to S. Porret-Blanc \cite{PB} to show that there is a critical intensity for the visibility to infinity in a Poisson line process in $\H^2$. The decay of the distribution tail of the total visibility differs from the Euclidean case which has been studied before \cite{Cal}.

The rest of the paper is organized as follows. In
Section~\ref{notation} we introduce the model of main interest
more carefully and state the main result, Theorem~\ref{maintthm1}.
In Section~\ref{pfmain} we give the proof of
Theorem~\ref{maintthm1}. We then discuss some extensions of the
main results to other models, in particular the line process model
in Section~\ref{extensions}. Section~\ref{nearcritical} provides the behaviour of the total visibility near the critical point and for small intensity. In the last section, we show an asymptotic result when the size of the balls goes to zero and the intensity increases accordingly.
\section{Notation and main results}\label{notation} Before turning to our
results, we introduce the model more carefully. We will work in
the Poincar\'e disc model of $\H^2$. This is the unit disc $\{z\in
{\mathbb C}\,:\,|z|<1\}$ equipped with the metric
$$ds^2=4\frac{dx^2+dy^2}{(1-(x^2+y^2))^2}.$$ M\"obius transforms
are isometries of $\H^2$, see~\eqref{mobius} below. The associated area measure $\mu$ is isometry-invariant:
$$\mu(dx,dy)=\frac{4}{(1-(x^2+y^2))^2}dxdy.$$
For more
information about hyperbolic geometry, we refer to \cite{flavors}.
Let us now describe the bounded radius version of the
Poisson-Boolean model of continuum percolation. We consider a homogeneous Poisson point process $X$ in $\H^2$, i.e. with intensity measure $\lambda \mu$ where $\lambda\in (0,\infty)$.
Let $C\in
(0,\infty)$ and suppose $R$ is a random variable with $R\in (0,C]$
a.s. Let
$$
\OC:=\bigcup_{x\in X} B(x,R_x)
$$
denote the {\it occupied set}, where $B(x,r)$ denotes the closed
ball of radius $r$ centered at $x$ and $\{R_x\}_{x\in X}$ is a collection of i.i.d. random variables with the same distribution as $R$. Let
$$
\VC:=\overline{\H^2\setminus\OC}.
$$
$\VC$ will be called the {\it vacant set}. It is well known that
both $\OC$ and $\VC$ satisfy the property of \emph{positive
correlations}, see Theorem 2.2 in \cite{meester}.
For $\VC$, this means that for any pair  $f$ and $g$ of bounded increasing functions of $\VC$, we have $\E[f(\VC) g(\VC)]\ge \E[f(\VC)]\E[g(\VC)]$, and the definition for $\OC$ is analogous.

This model has been extensively studied in Euclidean space, see in
particular \cite{Hall} and \cite{meester}.
Aspects of the
model have also been recently studied in hyperbolic space, see \cite{T} and
\cite{BJST}. We will soon mention some of the results in
\cite{BJST}, but first we introduce some notation.

For a set $A\subset \H^2$, let $A^R$ denote the closed
$R$-neighborhood of $A$: $$A^R=\{x\,:\,d(x,A)\le R\}.$$

With $c$ and $c'$ we denote positive constants and their values
may change from place to place, which may only depend on
$\lambda$, the law of $R$, and the parameter $\eps$ which is
introduced in Section~\ref{pfmain}. If they depend on some other
parameter, this is indicated. With  $\Theta(g)$ we denote a
quantity which takes its values between $c g$ and $c' g$. In
addition, we define $\tilde{\Theta}(g)$ in the same way as
$\Theta(g)$, but with condition that $c$ and $c'$ may not depend
on $\lambda$.

Let $L_r(\theta)$ be the geodesic line segment started at $0$ of
length $r$ such that its continuation hits $\partial \H^2$ at the
point $e^{i\theta}$.
For $\theta\in [0,2\pi)$, the visibility in direction $\theta$ is defined as $$V(\theta)=\inf\{r\ge 0\,:\,L_r(\theta)\cap \OC\neq \emptyset\}.$$ The total visibility is defined to be
$${\mathfrak V}=\sup_{\theta\in[0,2\pi)}V(\theta).$$

Let $f(r)=f_{\lambda,R}(r)$ be the probability that a line segment
of length $r$ is contained in $\VC$. Lemma 3.4 in \cite{BJST} says
that there is a unique $\alpha\ge 0$ such that
\begin{equation}\label{deceq}f(r)=\Theta(e^{-\alpha r}),\mbox{ }r\ge
0.\end{equation} The constant $\alpha$ depends on the law of $R$
and on $\lambda$ and it can be computed explicitly, we will come
back to this later. One of the main results in \cite{BJST} was the
following:

\begin{theorem}\label{bjsttheorem}
For the total visibility ${\mathfrak V}$ the following holds:
\begin{equation}
\left\{\begin{array}{ll}
\P[{\mathfrak V}=\infty]=0, & \alpha\ge 1 \\ \P[{\mathfrak V}=\infty]>0,
& \alpha<1
  \end{array} \right.
\end{equation}
\end{theorem}

We remark that in \cite{BJST}, Theorem~\eqref{bjsttheorem} was formulated in a much more general form. For example, visibility inside $\OC$ was also dealt with.

In \cite{BJST}, the decay of $\P[{\mathfrak V}\ge r]$ as $r\to \infty$ in the case $\alpha\ge 1$ was not studied. One of the main results in this note provides upper and lower bounds as follows:
\begin{theorem}\label{maintthm1}
For all $r$ large enough,
\begin{equation}
\left\{\begin{array}{ll} \P[{\mathfrak V}\ge r]=\Theta(1/r), &
\alpha= 1 \\ \P[{\mathfrak V}\ge r]=\Theta(e^{-(\alpha-1)r}), &
\alpha>1.
  \end{array} \right.
\end{equation}
\end{theorem}

\subsection{The value of $\alpha$}

If $R$ is non-random, then Lemma 4.2 in \cite{BJST} says that
$\alpha=2\lambda\sinh(R)$. We can easily generalize this to the
case when $R$ is random. For the convenience of the reader, we
include the proof. For this particular result, we do not need the
$R$ to be bounded.

\begin{lemma}\label{alphalemma}
If $R$ is random with $\E[e^R]<\infty$, then the value of $\alpha$
is given by
$$\alpha=2\lambda\E[\sinh(R)].$$
\end{lemma}
\begin{proof}
Let $\gamma\,:\,\R\rightarrow \H^2$ be a line parameterized by
arclength and let $r>0$. Let $\tilde{X}\subset X$ be the set of
Poisson points $x\in X$ for which $B(x,R_x)\cap
\gamma[0,r]\neq\emptyset$. If a Poisson point is at distance $t$
from $\gamma[0,r]$, the probability that its corresponding ball
intersects $\gamma[0,r]$ is equal to $\P[R\ge t]$. Therefore,
$\tilde{X}$ is a non-homogeneous Poisson point process with
intensity function $\Lambda(x)=\lambda \P[R\ge d(x,\gamma[0,r])]$.
That is, for any measurable $A\subset \H^2$,
\begin{equation}\P[|\tilde{X}\cap A|=k]=e^{-\int_A
\Lambda(x)\,d\mu(x)}\frac{\left(\int_A
\Lambda(x)\,dx\right)^k}{k!}.\end{equation}
\end{proof}
Observe that $\gamma[0,r]\subset\VC$ if and only if
$\tilde{X}=\emptyset$. Consequently, using Fubini,
\begin{equation*}
\begin{split}
f(r)&=\P[|\tilde{X}|=0]=e^{-\int_{\H^2}\Lambda(x)\,d\mu(x)}=e^{-\lambda\int_{\H^2}
\P[R\ge d(x,\gamma[0,r])]\,d\mu(x)}\\&=e^{-\lambda\int_{\H^2}\int
{\bf 1}\{R\ge d(x,\gamma[0,r])\}\,d\P \,d\mu(x)}=e^{-\lambda\int
\int_{\H^2}{\bf 1}\{R\ge d(x,\gamma[0,r])\}
\,d\mu(x)\,d\P}\\&=e^{-\lambda\int\mu(\gamma[0,r]^R)\,d\P}=e^{-\lambda\E[\mu(\gamma[0,r]^R)]}=e^{-\lambda\E[2\pi(\cosh(R)-1)+2r\sinh(R)]},
\end{split}
\end{equation*}
and the result follows. In the last equality, we used the calculation in the proof of Lemma 4.2 in \cite{BJST}.
\qed
\section{Proof of Theorem~\ref{maintthm1}}\label{pfmain}
We now turn to the proof of Theorem~\ref{maintthm1}. First we introduce some additional notation.
For $\eps,\delta\in [0,2\pi)$ let
$Y_r(\eps,\delta)$ be the set
$\{\theta\in[\eps,\delta]\,:\,L_r(\theta)\subset \VC\}$. Note that
a.s., $Y_r(\eps,\delta)$ is a union of intervals. Let
$y_r(\eps,\delta):=\length(Y_r(\eps,\delta))$.
Also put $Y_r(\eps)=Y_r(0,\eps)$, $y_r(\eps)=y_r(0,\eps)$,
$Y_r:=Y_r(2\pi)$ and $y_r:=y_r(2\pi)$.
Recall that $f(r)$ is the probability that a line segment of
length $r$ is contained in $\VC$. Since the law of $\VC$ is
invariant under isometries of $\H^2$, we have $f(r)=\P[\theta\in
Y_r]$, for every $\theta\in[0,2\pi)$. For $x,y\in\H^2$, let
$[x,y]$ be the line-segment between $x$ and $y$ and for $s>0$ let
$[x,y]_s$ be the union of all line-segments with one end-point in
$B(x,s)$ and the other end-point in $B(y,s)$. Let $Q(x,y,s)$ be
the event that $[x,y]_s\subset \VC$.

Clearly, \begin{equation}\label{obveq}f(d(x,y))\ge \P[Q(x,y,s)].\end{equation} However, from Lemma 3.3 in
\cite{BJST}, we have that there exists some $c_1>0$ such that for
all small enough $s$ and all $x,y\in \H^2$,
\begin{equation}\label{comparable}
 \P[Q(x,y,s)]\ge c_1 f(d(x,y)).
\end{equation}
If $R$ is fixed and one considers only intensities $\lambda$
within some compact interval, then $c_1$ can be chosen to be the
same for all those values of $\lambda$, and we will make use of
this later. In fact, Lemma 3.3 in \cite{BJST} states
relation~\eqref{comparable} for a larger class of random sets. We
can remove the condition that $s$ is small enough, as we will see
in the next lemma.

\begin{lemma}\label{interm}
For any $s\in(o,\infty)$ there is $c(s)>0$ such that for all $x,y\in \H^2$,
\begin{equation}\label{allseq}
f(d(x,y))\le c(s)\,\P[Q(x,y,s)].
\end{equation}
\end{lemma}

\begin{proof}
First we fix $s'>0$ so small that~\eqref{comparable} holds with $s'$ in place of $s$. Let $s\in(s',\infty)$. Let $\gamma$ be a line parameterized by arc-length. Fix $r>0$ large, and let $$t_1:=\inf\{t\,:\,d(\gamma(t),\partial [\gamma(0),\gamma(r)]_s)<s'\},$$
From Lemma 3.2 in \cite{BJST}, we get that for each $s>0$, there is some $c'(s)<\infty$ which is independent of $r$ such that $t_1<c'(s)$. In particular, $t_1$ does not diverge with $r$. Observe that by definition of $t_1$, \begin{multline}Q(\gamma(0),\gamma(r),s)\\ \supset Q(\gamma(0),\gamma(t_1),s)\cap Q(\gamma(t_1),\gamma(r-t_1),s')\cap Q(\gamma(r-t_1),\gamma(r),s).\end{multline} By positive correlations and invariance, we get that \begin{align}\label{e:qcomp}\P[Q(\gamma(0),\gamma(r),s)]\ge \P[Q(\gamma(0),\gamma(t_1),s)]^2 \P[ Q(\gamma(t_1),\gamma(r-t_1),s')]. \end{align} We have \begin{equation}\label{qequa1}\P[Q(\gamma(t_1),\gamma(r-t_1),s')]\ge \P[Q(\gamma(0),\gamma(r),s')]\stackrel{~\eqref{comparable}}{\ge} c_1 f(r)\end{equation} and \begin{equation}\label{qequa2}\P[Q(\gamma(0),\gamma(t_1),s)]^2\ge \P[Q(\gamma(0),\gamma(c'(s)),s)]^2.\end{equation}
We now deduce~\eqref{allseq} with $c(s)=c_1 \P[Q(\gamma(0),\gamma(c'(s)),s)]^2$ from~\eqref{e:qcomp}, ~\eqref{qequa1} and~\eqref{qequa2}.\qed
\end{proof}

Equations~\eqref{allseq} and~\eqref{obveq} together imply that for all $x,y\in \H^2$, \begin{equation}\label{fqcomp}\P[Q(x,y,s)]=\Theta(f(d(x,y))),\end{equation} where the implied constants in this case are allowed to depend on $s$.  Theorem \ref{maintthm1} is equivalent to the following estimate:
\begin{equation}\label{mainstat1}
\P[Y_r\neq \emptyset]=\left\{\begin{array}{ll}
\Theta\left(e^{-(\alpha-1)r}\right), & \alpha>1 \\ \Theta\left(r^{-1}\right)
& \alpha=1
  \end{array} \right.
\end{equation}
Recall that in \cite{BJST}, it is shown that in the case
$\alpha<1$, there is positive probability that there are infinite
rays contained in $\VC$ emanating from $0$, so that $\P[Y_r\neq
\emptyset]$ does not converge to $0$. We will make further remarks about the region $\alpha<1$ later. Observe that by Fubini, we
have
\begin{equation}\label{fubini1}\E[y_r(\eps)]=\eps\,\P[0\in
Y_r]=\eps\,f(r)\end{equation} and
\begin{equation}\label{fubini2}\E[y_r(\eps)^2]=\int_0^\eps
\int_0^\eps \P[\theta\in Y_r,\theta'\in
Y_r]\,d\theta\,d\theta'.\end{equation} Moreover, by invariance it
follows that $$\P[\theta\in Y_r,\theta'\in Y_r]=\P[0\in
Y_r,|\theta-\theta'|\in Y_r].$$ Therefore, we have

\begin{multline}\label{secondmomentest}\frac{\eps}{2}\int_{0}^{\eps/2}\P[0\in
Y_r,\theta\in Y_r]\,d\theta\le \int_0^\eps \int_0^\eps
\P[\theta\in Y_r,\theta'\in Y_r]\,d\theta\,d\theta'\\ \le
2\,\eps\int_{0}^{\eps}\P[0\in Y_r,\theta\in Y_r]\,d\theta.
\end{multline} Denote by $J(r,\theta)$ the set $L_r(0)\cup L_r(\theta)$. Note that $$\P[0\in Y_r,\theta\in Y_r]=e^{-\lambda\E[\area(J(r,\theta)^R)]}.$$ Since the area of
$J(r,\theta)^R$ is increasing in $\theta$ on $[0,\pi]$, it follows
that $\P[0\in Y_r,\theta\in Y_r]$ is decreasing in $\theta$
on $[0,\pi]$ and therefore we have
\begin{equation}\label{secondmomentest1}\int_{0}^{\eps/2}\P[0\in
Y_r,\theta\in Y_r]\,d\theta\ge \frac{1}{2}\int_{0}^{\eps}\P[0\in
Y_r,\theta\in Y_r]\,d\theta.\end{equation} Observe that up to a
set of measure $0$, the events $\{y_r>0\}$ and
$\{Y_r\neq\emptyset\}$ are the same. The following lemma is the
key ingredient for the proof of Theorem \ref{maintthm1}.

\begin{lemma}\label{tightsecondmoment}
For every $r>0$ and every
$\eps\in(0,\pi/2)$,
\begin{equation}\label{andramoment}
 \frac{\E[y_r(\eps)]^2}{\E[y_r(\eps)^2]}\le \P[Y_r(\eps)\neq\emptyset]\le 4 \frac{\E[y_r(\eps)]^2}{\E[y_r(\eps)^2]}.
\end{equation}
\end{lemma}
For the proof of Lemma \ref{tightsecondmoment} we will use some
techniques from \cite{J} and \cite{K}.
\begin{proof}
The lower bound is of course the usual second moment method, so it
remains to show the upper bound. The first part of the proof of
the upper bound follows the method in the proof of the Lemma on
page 146 of \cite{K}. Fix some $\eps\in(0,\pi/2)$. Let
$A=A(r,\eps)$ be the event that $Y_r(\eps)\neq\emptyset$. First we
show that
\begin{equation}\label{andramoment1}
\E[y_r(2 \eps)]\ge \P[A]\E[y_r(\eps)|0\in Y_r(\eps)],
\end{equation}
and then we deduce~\eqref{andramoment} from~\eqref{andramoment1}.
Let $A_N=A_N(r,\eps)$ be the event that $Y_r(\eps)$ contains an
interval of length at least $1/N$. Then clearly
$\P[A_N]\nearrow\P[A]$ as $N\nearrow\infty.$ Fix an integer $N$.
Let $A_0:=\{0\in Y_r\}$ and for $j=1,...,[N\eps]$ let
$$A_j:=\{0\in Y_r^c,1/N\in Y_r^c,...,(j-1)/N\in Y_r^c,j/N\in
Y_r\}.$$ On $A_N$, exactly one of the events $A_j$ happens. We
first argue that
\begin{equation}\label{andramoment2}\E[y_r(2\eps){\bf 1}_{A_j}]\ge\E[y_r(j/N,j/N+\eps){\bf
1}_{A_j}]\ge\P[A_j]\E[y_r(\eps)|0\in Y_r].\end{equation} The left
inequality is trivial. After division by $\P[A_j]$ we see that we
need to show that
$$\E[y_r(j/N,j/N+\eps)|A_j]\ge \E[y_r(\eps)|0\in
Y_r].$$ By invariance, the right hand side equals
$$\E[y_r(j/N,j/N+\eps)|j/N\in Y_r].$$ Thus it will suffice to
show that for each $\theta\in [j/N,j/N+\eps]$,
\begin{equation}\label{andramoment3}\P[\theta\in Y_r| A_j]= \P[\theta\in Y_r|j/N\in
Y_r].\end{equation} So fix some $\theta\in [j/N,j/N+\eps]$. We can
write
\begin{equation}\label{decomp}
X=\bigcup_{i=1}^4 X_i,
\end{equation}
where
\begin{equation}\label{decomp1}
X_1=\{x\in X\,:\,B(x,R_x)\cap L_r(j/N)\neq\emptyset\}
\end{equation}
\begin{equation}\label{decomp2}
X_2=\{x\in X\,:\,B(x,R_x)\cap L_r(j/N)=\emptyset,\,B(x,R_x)\cap
L_r(\theta)\neq\emptyset\}
\end{equation}
\begin{equation}\label{decomp3}
X_3=\{x\in X\,:\,B(x,R_x)\cap L_r(j/N)=\emptyset,\,B(x,R_x)\cap
\left(\cup_{i=0}^{j-1}L_r(i/N)\right)\neq\emptyset\}
\end{equation}
\begin{equation}\label{decomp4}
X_4=X\setminus\bigcup_{i=1}^3 X_i.
\end{equation}
 Note that $\{j/N \in
Y_r\}=\{X_1=\emptyset\}$. Therefore, given that the event $\{j/N
\in Y_r\}$ happens, the event $\{\theta\in Y_r\}$ is determined by
$X_2$, and the event
$$\tilde{A}:=\{0\in Y_r^c,1/N\in Y_r^c,...,(j-1)/N\in Y_r^c\}$$ is
determined by $X_3$. Therefore, conditioned on $\{j/N\in Y_r\}$,
the events $\tilde{A}$ and $\{\theta\in Y_r\}$ are conditionally
independent, that is
$$\P[\tilde{A}\cap \{\theta\in Y_r\}|j/N\in
Y_r]=\P[\tilde{A}|j/N\in Y_r]\P[ \theta\in Y_r|j/N\in Y_r].$$ This
implies that
$$\P[\theta\in Y_r|\tilde{A}\cap \{j/N\in
Y_r\}]=\P[\theta\in Y_r|j/N\in Y_r]$$ which is the same
as~\eqref{andramoment3} and therefore~\eqref{andramoment2} is
established. After summing both sides of~\eqref{andramoment2}, we get
\begin{equation}\label{andramoment4}\E[y_r(2\eps)]\ge\P[A_N]\E[y_r(\eps)|0\in Y_r].\end{equation}

Letting $N\to \infty$ in~\eqref{andramoment4}
establishes~\eqref{andramoment1}. We next show, in a similar way
as is done in \cite{J}, that
\begin{equation}\label{andramoment5}\E[y_r(\eps)|Y_r(\eps)\neq\emptyset]\ge\frac{1}{2}\E[y_r(\eps)|0\in
Y_r].\end{equation} This follows from
\begin{multline*}\E[y_r(\eps)|Y_r(\eps)\neq\emptyset]=\frac{\E[y_r(\eps)]}{\P[Y_r(\eps)\neq\emptyset]}=\frac{\E[y_r(2\eps)]/2}{\P[Y_r(\eps)\neq\emptyset]}\\ \overset{~\eqref{andramoment1}}\ge
\frac{\P[Y_r(\eps)\neq\emptyset]\E[y_r(\eps)|0\in
Y_r]/2}{\P[Y_r(\eps)\neq\emptyset]}=\frac{1}{2}\E[y_r(\eps)|0\in
Y_r],\end{multline*} where the second equality follows from
invariance. We can now derive the upper bound
in~\eqref{andramoment}:
\begin{multline*}
\P[Y_r(\eps)\neq\emptyset]=\frac{\E[y_r(\eps)]}{\E[y_r(\eps)|Y_r(\eps)\neq\emptyset]}\overset{~\eqref{andramoment5}}\le \frac{2 \E[y_r(\eps)]}{\E[y_r(\eps)|0\in Y_r]}\\=\frac{2 \E[y_r(\eps)]}{\int_0^{\eps}\P[\theta\in Y_r|0\in Y_r]\,d\theta}=\frac{2 \E[y_r(\eps)]\P[0\in Y_r]}{\int_0^{\eps}\P[\theta\in Y_r,0\in Y_r]\,d\theta}\overset{~\eqref{fubini1},~\eqref{secondmomentest}}\le 4\frac{\E[y_r(\eps)]^2}{\E[y_r(\eps)^2]},
\end{multline*}concluding the proof of the lemma.\qed
\end{proof}
{\bf Proof of Theorem \ref{maintthm1}.} In view of Lemma
\ref{tightsecondmoment}, we need to estimate $\E[y_r(\eps)^2]$.
First we estimate $\P[0\in Y_r,\theta \in Y_r]$ for
$\theta\in(0,\eps]$ and $r>0$.  We have $$ \P[0\in Y_r,\theta \in
Y_r]=\P[J(r,\theta)\subset \VC].$$ Let
$$t_\theta:=\inf\{t\,:\,d(L_\infty(\theta)\setminus
L_t(\theta),L_\infty(0))\ge 2C\}.$$ That is, if a point $x\in
L_\infty(\theta)$ is at distance more than $t_\theta$ from the
origin, the distance from $x$ to $L_\infty(0)$ is greater than or
equal to $2C$ (recall that if $d(A,B)\ge 2C$ then $A\cap\VC$ and $B\cap \VC$ are independent). Below, we will consider events of the type $\{L_r(0)\setminus L_s(0)\subset \VC\}$, and if $s\ge r$ we will use the convention that such an event is the entire sample space.

We get that
\begin{multline}\label{est1}
\P[J(r,\theta)\subset\VC]\\ =\P[\{J(r\wedge t_\theta,\theta)\subset\VC\}\cap \{L_r(0)\setminus L_{t_\theta}(0)\subset \VC\}\cap \{L_r(\theta)\setminus L_{t_\theta}(\theta)\subset \VC\}]\\ \ge\P[J(r\wedge t_\theta,\theta)\subset\VC]\P[L_r(0)\setminus L_{t_\theta}(0)\subset \VC]\P[L_r(\theta)\setminus L_{t_\theta}(\theta)\subset \VC],
\end{multline}
where the inequality follows from positive correlations.
On the other hand,
\begin{multline}\label{est2}
\P[\{J(r\wedge t_\theta,\theta)\subset\VC\}\cap \{L_r(0)\setminus L_{t_\theta}(0)\subset \VC\}\cap \{L_r(\theta)\setminus L_{t_\theta}(\theta)\subset \VC\}]\\ \le \P[\{J(r\wedge t_\theta,\theta)\subset\VC\}\cap \{L_r(0)\setminus L_{t_\theta+2C}(0)\subset \VC\}\cap \{L_r(\theta)\setminus L_{t_\theta+2C}(\theta)\subset \VC\}]\\=\P[J(r\wedge t_\theta,\theta)\subset\VC]\P[L_r(0)\setminus L_{t_\theta+2C}(0)\subset \VC]\P[L_r(\theta)\setminus L_{t_\theta+2C}(\theta)\subset \VC]\\=\Theta(1)\P[J(r\wedge t_\theta,\theta)\subset\VC]\P[L_r(0)\setminus L_{t_\theta}(0)\subset \VC]\P[L_r(\theta)\setminus L_{t_\theta}(\theta)\subset \VC],
\end{multline}
where we used independence at distance $2C$ in the first equality.
We also have
\begin{equation}\label{est3}
\P[J(t_{\theta}\wedge r)\subset \VC]\le \P[L_{t_{\theta}\wedge r}(\theta)\subset \VC].
\end{equation}
Let $x(l)$ be the point on $L_\infty(0)$ which is at distance $l$ from $o$. Then thanks to Lemma \ref{interm}, we have
\begin{multline}\label{est4}
\P[J(t_{\theta}\wedge r)\subset \VC]\ge \P[Q(o,x(t_{\theta}\wedge r),2C)]\overset{~\eqref{comparable}}=\Theta(1)\P[L_{t_{\theta}\wedge r}(0)\subset\VC].
\end{multline}
From~\eqref{est1}, ~\eqref{est2}, ~\eqref{est3} and ~\eqref{est4} we get
\begin{multline}\label{est5}
\P[J(r,\theta)\subset\VC]\\ =\Theta(1)\P[L_{t_{\theta}\wedge r}(0)\subset\VC]\P[L_r(0)\setminus L_{t_\theta}(0)\subset \VC]\P[L_r(\theta)\setminus L_{t_\theta}(\theta)\subset \VC]\\=\Theta(1)f(t_{\theta}\wedge r) f(0\vee r-t_\theta)^2=\Theta(1) f(r) f(0\vee r-t_\theta).
\end{multline}
Consequently,
\begin{equation}\label{andraest1}\int_0^{\eps}\P[0\in Y_r,\theta
\in Y_r]\,d\theta =\Theta(1)f(r)\int_0^{\eps}f(0\vee
r-t_\theta)\,d\theta.
\end{equation}
 We next investigate the behavior of $t_\theta$. Let $\gamma(t)$ be
the geodesic which starts at $0$ and then follows
$L_\infty(\theta)$, and suppose that $\gamma(t)$ is parameterized
by arc-length. Given $\theta$ and $t$, we first want to find the
distance between $\gamma(t)$ and $L_{\infty}(0)$. Denote this
distance by $s=s(t)$. By the hyperbolic law of cosines we have
\begin{equation}\label{hypcos}\cosh(2s)=\cosh^2(t)-\sinh^2(t)\cos(2\theta).\end{equation}
Using the identity $\cosh^2(t)-\sinh^2(t)=1$ we see that $s=2R$ if
and only if
\begin{align}\label{teq}t =t_\theta=
\cosh^{-1}\left(\sqrt{\frac{\cosh(4C)-\cos(2\theta)}{1-\cos(2\theta)}}\right).\end{align}
A calculation shows that $r-t_\theta>0$ if and only if
\begin{multline}\theta>
h(C,r):=\frac{1}{2}\cos^{-1}\left(\frac{\cosh^2(r)-\cosh(4C)}{\cosh^2(r)-1}\right)\\
=\frac{1}{2}\cos^{-1}\left(1-\frac{\cosh(4C)-1}{\cosh^2(r)-1}\right).\end{multline}
Now note that
\begin{equation}\label{hequat} h(r,C)=\Theta(1)\,e^{-r}\end{equation} for all $r$
large enough (using $\cos^{-1}(1-x)=\sqrt{2 x}+O(x^{3/2})$ for
small $x$). Let
$$\hat{t}(\theta):=\sqrt{\frac{\cosh(4C)-\cos(2\theta)}{1-\cos(2\theta)}}.$$

Using 
$$\cosh^{-1}(x)=\log(x+\sqrt{x^2-1})\in [\log(x),\log(x)+\log(2)),\quad x\ge 1$$
and
$$1-\cos(2\theta)=2\theta^2+O(\theta^3)\in[\frac{4}{\pi}\theta^2,2\theta^2],\quad \theta\in [0,\pi/4],$$
we get that for all large $r$
\begin{multline}\int_0^{\eps}f(0\vee r-t_\theta)\,d\theta\\
=\Theta(1)\left(\int_0^{h(C,r)}\,d\theta +
\int_{h(C,r)}^{\eps}e^{-\alpha(r-t_\theta)}\,d\theta\right)\\=\Theta(1)\left(h(C,r)+e^{-\alpha
r}\int_{h(C,r)}^{\eps}e^{\alpha(t_\theta)}\,d\theta\right)\\=
\Theta(1)\left(h(C,r)+e^{-\alpha
r}\int_{h(C,r)}^{\eps}\left(\hat{t}(\theta)+\sqrt{\hat{t}(\theta)^2-1}\right)^{\alpha}\,d\theta\right)
\end{multline}
\begin{multline}\label{longeq}
=\Theta(1)\left(h(C,r)+e^{-\alpha
r}\int_{h(C,r)}^{\eps}(1-\cos(2\theta))^{-\alpha/2}\,d\theta\right)
\\=\Theta(1)\left(e^{-r}+e^{-\alpha r}\int_{\Theta(1) e^{-r}}^{\eps}\theta^{-\alpha}\,d\theta\right)\\=\left\{\begin{array}{ll} \Theta(1)\,e^{-r}, & \alpha>1 \\ \Theta(1)\,r\,e^{-r}, & \alpha=1
  \end{array} \right.\end{multline}
Combining~\eqref{longeq}, \eqref{andraest1}, \eqref{secondmomentest} and \eqref{secondmomentest1}, we see that for large $r$,

\begin{equation}\label{andramomentfinal}
\E[y_r(\eps)^2]=\left\{\begin{array}{ll} \Theta(1)\,e^{-(1+\alpha)r}, & \alpha>1 \\ \Theta(1)\,r\,e^{-2 r}, & \alpha=1
  \end{array} \right.
\end{equation}

Note that
\begin{equation}\label{finaleqs}
\sum_{k=1}^8\P[Y_r((k-1)\pi/4,k\pi/4)\ne \emptyset]\ge \P[Y_r\ne
\emptyset]\ge
  \P[Y_r(\eps)\ne \emptyset].
\end{equation}
 Since $\E[y_r(\eps)]^2=\eps^2 e^{-2\alpha r}$,
Lemma~\ref{tightsecondmoment}, ~\eqref{andramomentfinal}
and~\eqref{finaleqs} implies that for large $r$,

\begin{equation}\label{asymptfinal}
\P[Y_r(\eps)\neq \emptyset]=\left\{\begin{array}{ll}
\Theta(1)\,e^{-(\alpha-1)r}, & \alpha>1 \\ \Theta(1)\,\frac{1}{r},
& \alpha=1.
  \end{array} \right.
\end{equation}
The result follows.\qed

\section{Generalizations of
Theorem~\ref{maintthm1}}\label{extensions} The proof of
Theorem~\ref{maintthm1} in Section~\ref{pfmain} can be fully or
partially adapted to other settings than visibility inside
${\mathcal W}$. Here are some important cases.
\subsection{Random convex shapes}
Let $K$ be a closed random convex shape which contains the origin, such that the diameter of
$K$ is a.s. less than $C<\infty$. In addition, assume that the law
of $K$ is invariant under all rotations of ${\mathbb H}^2$. For
$x\in \H^2$ let $\phi_x\,:\,\H^2\to \H^2$ be the M\"obius
transform mapping $x$ to $0$:
\begin{equation}\label{mobius}\phi_x(z)=\frac{z-x}{1-\bar{x}\,z}.\end{equation} For each $x\in X$, let
$K_x$ be an independent copy of $K$, and let
$${\mathcal C}_K=\bigcup_{x\in X}\phi_x^{-1}(K_x)\text{ and }{\mathcal W}_K=\overline{{\mathbb H}^2\setminus {\mathcal C}}.$$

It is easy to see that the proofs above for balls of random radius
are adaptable to this more general case. All results from \cite{BJST} used in the above proofs are valid also in this case. Thus the conclusions
of Theorem~\ref{maintthm1} and Proposition~\ref{critprop} remain
true when replacing balls with random convex shapes. The value of
$\alpha$ will of course depend on the law of $K$. In this case one
gets, as in the proof of Lemma~\ref{alphalemma},
\begin{equation}\label{alphaconvex}f(r)=e^{-\lambda\E[\mu(\{x\,:\,\phi_x^{-1}(K)\cap\gamma[0,r]\neq\emptyset\})]}.\end{equation}
To find the explicit value of $\alpha$, one has to calculate the
expectation appearing in the exponent in~\eqref{alphaconvex}.

\subsection{Asymptotics of visibility in the covered set}
It is also of interest to consider visibility inside the covered set ${\mathcal C}$. Let ${\mathfrak V}'$ be the supremum of the set of $r\ge 0$ such that there is a line-segment of length $r$ starting at the origin which is fully contained in ${\mathcal C}$. Let $h(r)$ be the probability that a fixed line-segment is contained in ${\mathcal C}$. In \cite{BJST}, it was shown that there is some $\alpha'$ such that $h(r)=\Theta(e^{-\alpha' r})$. The lower bound in Theorem~\ref{maintthm1} is just using the ordinary second moment method. Moreover, the calculations in the proof of Theorem~\ref{maintthm1}, might be adapted to visibility inside ${\mathcal C}$, except where reference to Lemma~\ref{tightsecondmoment} is made. The derivation of the upper bound in Lemma~\ref{tightsecondmoment} does not go through for the covered set. In particular, we currently do not know how to prove Eq.~\eqref{andramoment3}. Consequently, at the moment we only know lower bounds as follows. There is $c>0$ (depending on the law of the obstacles) and $r_0<\infty$ such that

\begin{equation}\label{covset}
\left\{\begin{array}{ll}
\P[{\mathfrak V}'\ge r]\ge c\, r^{-1}, & \alpha'= 1,\text{ }r\ge r_0 \\ \P[{\mathfrak V}'\ge r]\ge c\, e^{-(\alpha'-1)r},
& \alpha'>1,\text{ }r\ge r_0.
  \end{array} \right.
\end{equation}

\subsection{Asymptotics of visibility outside a Poisson line
process} We consider a Poisson line process in the Poincar\'e disc
model of $\H^2$ defined as follows: we let ${\mathcal P}$ be a Poisson point
process in the open unit disk with intensity measure
$$\mu_{\lambda}(d\rho,d\theta)=2\lambda \frac{(1+\rho^2)}{(1-\rho^2)^2}d\rho d\theta.$$
For every $x\in {\mathcal P}$, let $G_x$ be the hyperbolic line which contains $x$ and is orthogonal to the
Euclidean line segment $[0,x]$. Let \begin{align}\label{linedef}{\mathcal L}=\bigcup_{x\in {\mathcal P}}G_x.\end{align} In particular, the law of ${\mathcal L}$ is invariant under rotations around $0$, and this will be used below without further mention.

In the same spirit as for the Boolean model, we denote by
$Y_r(\epsilon)$, $\epsilon\in [0,2\pi]$, the set of all directions
$\theta\in [0,\epsilon)$ such that the line $L_r(\theta)$ does not intersect ${\mathcal L}$. We keep the
same notations $y_r(\epsilon)$ and $Y_r:=Y_r(2\pi)$. In other
words, $Y_r$ is the set of directions in which we can see up to
distance $r$ without meeting any line from the Poisson line
process.

In \cite{BJST}, the existence of an explicit critical intensity
equal to $\lambda=\frac{1}{2}$ has been proved (In \cite{BJST}, a different but equivalent, up to scaling of the intensity measure, way of describing the Poisson line process was used. Therefore, the critical value there is $1$, rather than $1/2$.) In \cite{PB}, an
upper-bound for the distribution tail of the maximal visibility
had previously been derived. We intend here to get a new more precise
estimate as in Theorem \ref{maintthm1}.

In particular, we can show an analogue of Lemma
\ref{tightsecondmoment}: for every $r>0$ and $\epsilon\in
(0,\pi/2),$
\begin{equation}\label{tsmline}
 \frac{\E[y_r(\eps)]^2}{\E[y_r(\eps)^2]}\le \P[Y_r(\eps)\neq\emptyset]\le 4 \frac{\E[y_r(\eps)]^2}{\E[y_r(\eps)^2]}.
\end{equation}
The proof of (\ref{tsmline}) can be written along the same lines.
The only point which requires an extra argument is the extention
of the equality (\ref{andramoment}) to the setting of the Poisson
line process. To do so, let us define
$$ M_r(\theta)=\{x\in \H^2: G_x\cap L_r(\theta)\ne \emptyset\}.$$
Then conditionally on $\{j/N\in Y_r\}$, the events $\{0\in
Y_r^c,1/N\in Y_r^c,...,(j-1)/N\in Y_r^c\}$ and $\{\theta\in Y_r\}$
are independent. Indeed, the first one is determined by the
intersection of the point process ${\cal P}$ with
$\bigcup_{i=0}^{j-1}M_r(i/N)\setminus M_r(j/N)$ whereas the second
one is determined by the intersection of ${\cal P}$ with a
disjoint set, namely $M_r(\theta)\setminus M_r(j/N)$. This is
sufficient to prove (\ref{andramoment}) and deduce
(\ref{tsmline}).

We now use (\ref{tsmline}) to show our main theorem.
\begin{theorem}\label{mainthmplp}
When $r\to \infty$, we have
\begin{equation}\label{mainstat3}
\P[Y_r\neq \emptyset]=\left\{\begin{array}{ll}
\Theta(1)\,e^{-(2\lambda-1)r}, & \lambda>1/2 \\
\Theta(1)\,\frac{1}{r} & \lambda=1/2
  \end{array} \right.
\end{equation}
\end{theorem}
\begin{proof}
As for Theorem \ref{maintthm1}, the proof relies on the use of
(\ref{tsmline}) and the estimation of both the first and second
moments of $y_r(\epsilon)$.

By equation 17.61 in \cite{Santalo}, we have
$$\P[0\in Y_r]=\exp(-2\lambda r).$$
Moreover $$\P[0\in Y_r,\theta\in
Y_r]=\exp(-\lambda\mbox{per}(T_{r,\theta}))$$ where $\mbox{per}$
denotes the perimeter and $T_{r,\theta}$ is the hyperbolic
triangle with apices $0$, $a_r$ and $b_r$, $a_r$ (resp. $b_r$)
being the point on $L_r(0)$ (resp. $L_r(\theta)$) at distance $r$
from the origin.

We have
$$\mbox{per}(T_{r,\theta})=2r+\cosh^{-1}(\cosh^2(r)(1-\cos(\theta))+\cos(\theta)).$$
In particular, since $\cosh^{-1}(t)=\log(t+\sqrt{t^2-1})$ for every $t\ge 1$, we have
$$\log(t)\le \cosh^{-1}(t)\le \log(t)+\log(2),\quad t\ge 1.$$
Consequently, we deduce that when $r\to \infty$,
\begin{eqnarray}\label{interm31}
\int_0^{\epsilon}\P[0\in Y_r,\theta\in Y_r]d\theta&=& \Theta(1)e^{-2\lambda r}\int_0^{\epsilon}\left(\cosh^2(r)(1-\cos(\theta))+\cos(\theta)\right)^{-\lambda}d\theta\nonumber\\
&=& \Theta(1)e^{-2\lambda r}\int_0^{\epsilon}(e^{2r}+2+\cos(\theta)(2-e^{2r}))^{-\lambda}d\theta.
\end{eqnarray}
Moreover, for any $\theta\in(0,\pi/2)$, 
$1-\frac{\theta^2}{2}\le\cos(\theta)\le 1-\frac{\theta^2}{\pi}.$ Replacing $\cos(\theta)$ in \eqref{interm31}, we notice that for $C$ equal to $2$ or $\pi$, we have
\begin{eqnarray}
 \int_0^{\epsilon}\left(e^{2r}+2+(1-\frac{\theta^2}{C})(2-e^{2r})\right)^{-\lambda}d\theta&=&\int_0^{\epsilon}\left(4+\theta^2(\frac{e^{2r}}{C}-\frac{2}{C})\right)^{-\lambda}d\theta\nonumber\\&&\nonumber\\
&=&\left(\frac{e^{2r}}{C}-\frac{2}{C}\right)^{-1/2}
\int_0^{\frac{\epsilon}{\sqrt{C}}\sqrt{e^{2r}-2}}
\frac{d\theta}{(4+\theta^2)^{\lambda}}
\nonumber\\&&\nonumber\\
&=& \left\{\begin{array}{ll}\Theta(1)e^{-r}\int_0^{\infty}\frac{d\theta}{(1+\theta^2)^{-\lambda}}&\mbox{ if
$\lambda>1/2$}\\\Theta(1)e^{-r}\cdot r&\mbox{ if
$\lambda=1/2$}.\end{array}\right.\nonumber
\end{eqnarray}
Inserting this last result in \eqref{interm31}, we obtain that
\begin{equation}\label{interm32}
\int_0^{\epsilon}\P[0\in Y_r,\theta\in Y_r]d\theta= \left\{\begin{array}{ll}\Theta(1)e^{-(2\lambda+1)
r}&\mbox{ if
$\lambda>1/2$}\\\Theta(1)e^{-2r}\cdot r&\mbox{ if
$\lambda=1/2$}.\end{array}\right.
\end{equation}
We conclude by inserting \eqref{interm32} in (\ref{tsmline}).
\end{proof}

\section{Critical point and small intensity}\label{nearcritical}
\subsection{When $\alpha\to 1$}
In this section, we study the behavior of the visibility near the critical point $\lambda=\lambda_{gv}$. Recall that ${\mathfrak V}$ is the total visibility, i.e.
$${\mathfrak V}=\sup\{r>0: Y_r\ne\emptyset\}.$$
Let $${\mathcal S}=\{x\in \H^2\,:\,[0,x]\subset\VC\}$$ be the set
of all points visible from the origin. The set ${\mathcal S}$ is
sometimes called the visibility star. Recall that
$\alpha=2\lambda\E[\sinh(R)]$.

\begin{proposition}\label{expmaxvis}
When $\lambda\searrow \lambda_{gv}$, we have
$$\E[\area({\mathcal S})]=\frac{\tilde{\Theta}(1)}{\alpha-1}\;\; \mbox{ and } \;\E[{\mathfrak V}]=\frac{\tilde{\Theta}(1)}{\alpha-1}.$$
\end{proposition}
\begin{proof}
Fix $\lambda_1>\lambda_{gv}$. We now verify that $\Theta(1)$ in both (\ref{deceq})
and in (\ref{mainstat1}) stay in $(0,\infty)$ when $\lambda\in[0,\lambda_1]$,
 which implies that they can be replaced by $\tilde{\Theta}(1)$ when $\lambda\in[0,\lambda_1]$.

Indeed, we get from the proof of Lemma 3.4 in \cite{BJST} that the
constant in $(\ref{deceq})$ is between $f(2C)$ and $1$ and the
quantity $f(2C)\in(0,1]$  for $\lambda\le \lambda_1$.

As for $\Theta(1)$ in (\ref{mainstat1}), it is deduced from
displays (\ref{est2}), (\ref{est4}) and (\ref{est5}).
\begin{itemize}
\item In (\ref{est2}), we have
  \begin{eqnarray*}
\P[L_r(0)\setminus L_{t_\theta}(0)\subset \VC]&=&\P[\{L_r(0)\setminus L_{t_\theta+2C}(0)\subset \VC\}\\ & & \cap\{L_{t_\theta+2C}(0)\setminus L_{t_\theta}(0)\subset \VC\}]\\
&\ge&\P[L_r(0)\setminus L_{t_\theta+2C}(0)\subset \VC]\\ & &\times \P[L_{t_\theta+2C}(0)\setminus L_{t_\theta}(0)\subset \VC]\\
&= & \P[L_r(0)\setminus L_{t_\theta+2C}(0)\subset \VC]\\ & & \times
\P[L_{2C}(0)\subset \VC].
  \end{eqnarray*}
It suffices to see now that $\P[L_{2C}(0)\subset
\VC]=f(2C)\in(0,1]$ for $\lambda\le \lambda_1$.
\item In (\ref{est4}), the constant comes from the calculation of $\P[Q(x,y,s)]$, $x,y\in\H^2$, $s>0$. Let us consider
\begin{equation}\label{defE[x,y]_s^R}
{\mathcal E}([x,y]_s,R)=\{z\in \H^2:B(z,R)\cap
[x,y]_s\ne\emptyset\}.
\end{equation}
Then
\begin{align}\begin{split}
\P[Q(x,y,s)]&=e^{-\lambda\cdot {\area(\mathcal E}([x,y]_s,R))}\\&=e^{-\lambda(\area({\mathcal E}([x,y]_s,R))-\area({\mathcal E}([x,y]_0,R)))}f(d(x,y)).\end{split}
\end{align}
and it remains to notice that $e^{-\lambda(\area({\mathcal
E}([x,y]_s,R))-\area({\mathcal E}([x,y]_0,R)))}\in(0,1]$ when $\lambda\le \lambda_1$.
\item In (\ref{est5}), the constant is between $f(4C)$ and $1$ because when $r>t_{\theta}$, we have
$$f(r)f(4C)\le f(t_{\theta})f(r-t_{\theta})\le f(r).$$
\end{itemize}
Now, a classical use of Fubini's theorem and~\eqref{alphalemma} yields
to
\begin{eqnarray*}
\E[\area({\mathcal S})]&=&2\pi\int_0^1f\left(\log\left(\frac{1+r}{1-r}\right)\right)\frac{4r dr}{(1-r^2)^2}\\
&=&\tilde{\Theta}(1)\int_0^1\left(\frac{1+r}{1-r}\right)^{\alpha}\frac{r dr}{(1-r)^2(1+r)^2}\\
&=&\tilde{\Theta}(1)\int_0^1\frac{dr}{(1-r)^{2-\alpha}}\\
&=&\frac{\tilde{\Theta}(1)}{\alpha-1}\mbox{ when
$\lambda\searrow \lambda_{gv}$.}
\end{eqnarray*}
In the same way, the second estimate is obtained with the use of
(\ref{mainstat1}).
$$\E[{\mathfrak V}]=\int_0^{\infty}\P[{\mathfrak V}\ge r]dr=\int_0^{\infty}\P[Y_r\ne\emptyset]dr=\frac{\tilde{\Theta}(1)}{\alpha-1}\text{, }\lambda\searrow \lambda_{gv}.$$
\end{proof}

We conclude the section by studying how the probability to see to infinity increases as $\lambda$ increases from $\lambda_{gv}$.

\begin{proposition}\label{critprop}For $\lambda\in [0,\lambda_{gv}]$,
\begin{equation}\label{critexp}
 \P[Y_\infty(\eps)\neq\emptyset]=\tilde{\Theta}(1)(1-\alpha)
\end{equation}
\end{proposition}
\begin{proof}
Repeating the calculations
leading to~\eqref{longeq} and~\eqref{andramomentfinal}, and using that from the proof of Proposition~\ref{expmaxvis} we can replace $\Theta(1)$ with $\tilde{\Theta}(1)$ at appropriate places, it follows that
\begin{equation}\label{supercrit1}
\int_0^{\eps}f(0\vee r-t_\theta)\,d\theta
=\tilde{\Theta}(1)\left(e^{-r}+\frac{e^{-\alpha
r}}{1-\alpha}-\frac{e^{-r}}{1-\alpha}\right)
\end{equation}
and
\begin{equation}\label{supcrit2}\E[y_r(\eps)^2]=\tilde{\Theta}(1)\left(e^{-(1+\alpha)r}+\frac{e^{-2\alpha
r}}{1-\alpha}-\frac{e^{-(1+\alpha)r}}{1-\alpha}\right)=\tilde{\Theta}(1)\frac{e^{-2\alpha r}}{1-\alpha},\end{equation}
where the last equality follows since $\alpha<1$. Using $\E[y_r(\eps)]^2=\eps^2 e^{-2\alpha r}$, Lemma~\ref{tightsecondmoment} and letting $r\to \infty$ we obtain the the result.\qed
\end{proof}
\subsection{When $\alpha\to 0$}
We conclude the section by showing that as $\alpha\to 0$, the probability to see to infinity from a given point goes to $1$.
\begin{proposition}
\begin{align}
\lim_{\alpha\to 0}\P[{\mathfrak V}=\infty]=1.
\end{align}
\end{proposition}
\begin{proof}
In view of~\eqref{andramoment}, it is enough to show
\begin{align}\label{e.ets}
\lim_{\alpha\to 0}\liminf_{r\to \infty}\frac{\E[y_r]^2}{\E[y_r^2]}=1.
\end{align}
We have
\begin{align}
\begin{split}
\P[\theta,\theta'\in Y_r] & \le \P[L_r(\theta)\subset \VC]\P[L_r(\theta')\setminus L_{t_{|\theta-\theta'|}}\subset \VC]\\
 & \le C(\alpha)^2e^{-\alpha r}e^{-\alpha(0\vee r-t_{|\theta-\theta'|})}\\ & \le C(\alpha)^2e^{-2\alpha r}e^{\alpha t_{|\theta-\theta'|}}
\end{split}
\end{align}
where $$C(\alpha)=\exp\left(-\frac{\alpha}{2\E[\sinh(R)]}\E[2\pi(\cosh(R)-1)]\right)\underset{\alpha\to 0}{\longrightarrow}1$$ is the constant that we obtained in the proof of Lemma \ref{alphalemma} when calculating $f(r)$.
Therefore,
\begin{align}
\E[y_r^2]\stackrel{~\eqref{fubini2}}{\le}C(\alpha)^2e^{-2\alpha r}\int_{0}^{\eps}\int_{0}^{\eps}e^{\alpha t_{|\theta-\theta'|}}d\theta\,d\theta'.
\end{align}
Since $\E[y_r]^2=C(\alpha)^2\eps^2 e^{-2\alpha r}$, it follows that
\begin{align}\label{e.smoment}
\frac{\E[y_r]^2}{\E[y_r^2]}\ge \frac{\eps^2}{\int_{0}^{\eps}\int_{0}^{\eps}e^{\alpha t_{|\theta-\theta'|}}d\theta\,d\theta'}
\end{align}
 Using~\eqref{teq}, we see that \begin{align}
t_{|\theta-\theta'|}\le \log\left(\frac{O(1)}{\sqrt{1-\cos(2|\theta-\theta'|)}}\right)\le \log\left(\frac{O(1)}{|\theta-\theta'|}\right).
\end{align}
Hence,
\begin{align}\label{e.alim}
\int_{0}^{\eps}\int_{0}^{\eps}e^{\alpha t_{|\theta-\theta'|}}d\theta\,d\theta'\le O(1)^{\alpha}\int_{0}^{\eps}\int_{0}^{\eps}|\theta-\theta'|^{-\alpha}d\theta d\theta'\stackrel{\alpha\to 0}{\to}\eps^2,
\end{align}
where the limit follows from straight forward calculations. Now~\eqref{e.ets} follows from~\eqref{e.smoment} and~\eqref{e.alim}.\qed
\end{proof}
\section{Visibility with varying intensity}
We consider the case where all radii are deterministic, equal to $R>0$. For a fixed intensity, there exists a critical radius $R_C=\sinh^{-1}\left(\frac{1}{2\lambda}\right)$ under which visibility to infinity occurs with positive probability. When the radius $R$ goes to $0$, this probability goes to $1$. The question we are interested in this section is the following: what happens when the intensity $\lambda$ of the underlying Poisson point process is a function $\lambda(R)$ of the radius which goes to infinity when $R\to 0$?

Let ${\mathfrak V}_{\lambda,R}$ be the total visibility associated with the choice of $R$ for the radius of the balls and $\lambda$ for the intensity of the underlying Poisson point process. In the following result, we show that we can adapt the intensity so that the maximal visibility will not be higher than a fixed level with high probability.
\begin{theorem}
For every $r>0$ and $p\in (0,1)$,
there exists an explicit functional $\lambda(R)$ given by (\ref{equallambda0})
such that
$\lim_{R\to 0} \P[{\mathfrak V}_{\lambda(R),R}\le r]=p.$
\end{theorem}
\begin{proof}
We denote by $\overline{r}=\tanh(r/2)$  and $\overline{R}=\tanh(R/2)$.
A ball $B_{\H}(x,R)$ intersects $B_{\H}(0,r)=B_{\R^2}(0,\overline{r})$ if and only if
$\|x\|\le \alpha(\overline{r})$ where for every $\ovr\in [0,1]$,
$$\alpha(\ovr)=\frac{\sqrt{(1-\ovR^2)^2+4(\ovR+\ovr)(\ovR+\ovR^2\ovr)}-(1-\overline{R}^2)}{2\ovR(1+\ovR\ovr)}.$$
The number of such $x$ is Poisson distributed of mean
\begin{equation}\label{eq:1}
2\pi\Lambda=4\lambda\pi\int_0^{\alpha(\overline{r})}\frac{2\rho}{(1-\rho^2)^2}d\rho=4\lambda\pi\frac{\alpha^2(\overline{r})}{1-\alpha^2(\overline{r})}.
\end{equation}
These points are independent, rotation-invariant and the common density of their radial coordinates is
\begin{equation}
  \label{eq:53}
f(\rho)=\frac{4\lambda}{\Lambda} {\bf 1}_{[0,\alpha(\overline{r})]}(\rho)\frac{\rho}{(1-\rho^2)^2}.
\end{equation}
In particular, the
(normalized) size $A_{\ovr,\ovR}$ of the 'shadow' of one such ball $B_{\H}(x,R)$
is equal to
\begin{equation}
  \label{eq:54}
\frac{1}{\pi}\arcsin\left(\ovR\frac{1-\|x\|^2}{\|x\|(1-\ovR)^2}\right)
\end{equation}
if $0\le\|x\|\le\beta(\overline{r})=\sqrt{\frac{\ovR^2+\ovr^2}{\ovR^2\ovr^2+1}}$
and something smaller if $\beta(\ovr)<\|x\|\le \alpha(\ovr).$
It is easy to check that when $R\to 0$, the probability that a $f$-distributed random variable is in $[\beta(\ovr),\alpha(\ovr)]$ goes to $0$.
Consequently, we may use the formula (\ref{eq:54}) combined with (\ref{eq:53}) to show that $A_{\ovr,\ovR}/\ovR$ converges in distribution to a limit distribution.

The probability of the event $\{{\mathfrak V}_{\lambda,R}\le r\}$ is equal to the probability to cover the Euclidean circle centered at the origin and of radius $\ovr$ by a Poisson number of mean $2\pi\Lambda$ of i.i.d. random arcs such that their normalized lengths are distributed as $A_{\ovr,\ovR}$. We are going to use a slightly modified version of an original result due
to Janson:
for every $\Lambda,\varepsilon>0$, let $p_{\Lambda,\varepsilon}$ be
the probability of covering the circle of perimeter one  with a Poisson number of mean $2\pi\Lambda$ of independent
 and uniformly located random arcs with a half-length distributed as
 $\varepsilon \widetilde{{\mathcal R}}_{\Lambda}$, $\widetilde{{\mathcal R}}_{\Lambda}$ being a bounded random variable for every $\Lambda$. If:
\begin{enumerate}
\item $\widetilde{{\mathcal R}}_{\Lambda} \to \widetilde{{\mathcal R}}$ in distribution as $\Lambda\to \infty$, where $\widetilde{{\mathcal R}}$ is a random variable with a finite moment of order $(1+\varepsilon)$ for some $\varepsilon>0$, and
\item $\varepsilon$ (going to $0$) and $\Lambda$ (going to $\infty$) are related such that the following convergence occurs:
  \begin{equation}
    \label{eq:2}
\lim_{\epsilon \to 0, \Lambda\to\infty
  }\left\{2\pi b{\varepsilon}\Lambda+
  \log(b{\varepsilon})-\log(-\log(b{\varepsilon}))
\right\}=t
  \end{equation}
where $b:=\frac{1}{\pi}\E[\widetilde{{\mathcal R}}],$
\end{enumerate}
then the probability $p_{\Lambda,\varepsilon}$ goes to
$\exp(-e^{-t})$.

We apply the above result with the choice $\e=\overline{R}$, $\Lambda$
given by (\ref{eq:1}) and $t$ such that $\exp(-e^{-t})=p$. We can verify that in this case $\widetilde{R}$ is distributed as $\frac{1-X^2}{X}$ (up to a multiplicative constant) where $X$ is $f$-distributed. In particular, $\E[\widetilde{R}^{1+\varepsilon}]<\infty$ for every $0\le \varepsilon<1$.
With the choice
\begin{equation}\label{equallambda0}
\lambda(R)=\frac{1-\alpha^{2}(\ovr)}{2\alpha^2(\ovr)}\left[-\frac{\log(\overline{R})}{2\pi b \overline{R}}+\frac{\log(-\log(\overline{R}))}{2\pi b \overline{R}}+\frac{t-\log(b)}{2\pi b\overline{R}}\right],
\end{equation}
we deduce from the covering result due to Janson that
$$\lim_{R\to 0} \P[{\mathfrak V}_{\lambda(R),R}\le r]=p.$$
\end{proof}
\bibliographystyle{plain}
\bibliography{tykesson}
\end{document}